\documentclass{amsart}

\usepackage{enumerate}
\newtheorem{theorem}{Theorem}[section]
\newtheorem{lemma}[theorem]{Lemma}

\theoremstyle{definition}

\theoremstyle{remark}
\newtheorem{remark}[theorem]{Remark}

\numberwithin{equation}{section}

\begin{document}

 \title[Polyharmonic Dirichlet and Neumann eigenvalues]{Inequalities between Dirichlet and Neumann eigenvalues of the Polyharmonic operators}
%\title{Inequalities between Dirichlet and Neumann eigenvalues of the Polyharmonic operators}


%    author one information
\author{Luigi Provenzano}
\address{Universit\`a degli Studi di Padova, Dipartimento di Matematica, Via Trieste 63, 35121 Padova, Italy}
\curraddr{}
\email{luigi.provenzano@math.unipd.it}
\thanks{The author is member of the Gruppo Nazionale per l'Analisi Matematica, la
Probabilit\`a e le loro Applicazioni (GNAMPA) of the Istituto Nazionale di Alta Matematica (INdAM)}

%    author two information
%\author{}
%\address{}
%\curraddr{}
%\email{}
%\thanks{}

%    \subjclass is required.
\date{}

\dedicatory{}

%    "Communicated by" -- provide editor's name; required.
\commby{}

%    Abstract is required.
\begin{abstract}
We prove that $\mu_{k+m}^m <\lambda_k^m$, where $\mu_k^m$ ($\lambda_k^m$) are the eigenvalues of $(-\Delta)^m$ on $\Omega\subset\mathbb R^d$, $d\geq 2$, with Neumann (Dirichlet) boundary conditions.
\end{abstract}

\maketitle


%%%%%%%%%%%%%%%%%%%%%%%%%%%%%%%%%%%%%%%%%%%%%%%%%%%%%%%%%%%%%%%%%%%%%%%%%%%%%%%%%%%%%%%%%%%%%%%%%%%%%%%%%%%%%%%%%%%%%%%%%%%%%%%%%%%%%%%%
%%%%%%%%%%%%%%%%%%%%%%%%%%%%%%%%%%%%%%%%%%%%%%%%%%%%%%%%%%%%%%%%%%%%%%%%%%%%%%%%%%%%%%%%%%%%%%%%%%%%%%%%%%%%%%%%%%%%%%%%%%%%%%%%%%%%%%%%
%%%%%%%%%%%%%%%%%%%%%%%%%%%%%%%%%%%%%%%%%%%%%%%%%%%%%%%%%%%%%%%%%%%%%%%%%%%%%%%%%%%%%%%%%%%%%%%%%%%%%%%%%%%%%%%%%%%%%%%%%%%%%%%%%%%%%%%%

%%%%%%%%%%%%%%%%%%%%%%%%%%%%%%%%%%%%%%%%%%%%                INTRODUCTION     AND MAIN             %%%%%%%%%%%%%%%%%%%%%%%%%%%%%%%%%%%%%%%%%%%%% 

%%%%%%%%%%%%%%%%%%%%%%%%%%%%%%%%%%%%%%%%%%%%%%%%%%%%%%%%%%%%%%%%%%%%%%%%%%%%%%%%%%%%%%%%%%%%%%%%%%%%%%%%%%%%%%%%%%%%%%%%%%%%%%%%%%%%%%%%
%%%%%%%%%%%%%%%%%%%%%%%%%%%%%%%%%%%%%%%%%%%%%%%%%%%%%%%%%%%%%%%%%%%%%%%%%%%%%%%%%%%%%%%%%%%%%%%%%%%%%%%%%%%%%%%%%%%%%%%%%%%%%%%%%%%%%%%%
%%%%%%%%%%%%%%%%%%%%%%%%%%%%%%%%%%%%%%%%%%%%%%%%%%%%%%%%%%%%%%%%%%%%%%%%%%%%%%%%%%%%%%%%%%%%%%%%%%%%%%%%%%%%%%%%%%%%%%%%%%%%%%%%%%%%%%%%

\section{Introduction and main result}

Let $m\in\mathbb N$. Let $\Omega$ be a bounded domain (i.e., a bounded connected open set) in $\mathbb R^d$, $d\geq 2$, such that the embedding $H^m(\Omega)\subset L^2(\Omega)$ is compact. We denote here by $H^m(\Omega)$ the standard Sobolev space of (complex valued) functions in $L^2(\Omega)$ with weak derivatives up to order $m$ in $L^2(\Omega)$, endowed with the standard scalar product
\begin{equation}
\langle u,v\rangle_{H^m(\Omega)}=\int_{\Omega}D^m u\cdot D^mv + u \bar v dx,
\end{equation}
and induced norm
\begin{equation}
\|u\|_{H^m(\Omega)}=\left(\int_{\Omega}|D^m u|^2 + |u|^2 dx\right)^{1/2},
\end{equation}
where
\begin{equation}
D^mu\cdot D^mv=\sum_{j_1,...,j_m=1}^d\left(\partial_{j_1\cdots j_m}^m u\right)\overline{\left(\partial_{j_1\cdots j_m}^m v\right)}.
\end{equation}
Here and in what follows $\partial^m_{j_1\cdots j_m}$ denotes $\frac{\partial^m}{\partial{x_{j_1}}\cdots\partial{x_{j_m}}}$. 

We consider the following variational problem:
\begin{equation}\label{problemV}
\int_{\Omega}D^mu\cdot D^m\phi dx=\Lambda\int_{\Omega}u\bar\phi dx\,,\ \ \ \forall\phi\in H(\Omega),
\end{equation}
in the unknowns $u\in H(\Omega)$ and $\Lambda\in\mathbb R$. Here $H(\Omega)$ denotes a subspace of $H^m(\Omega)$ containing $H^m_0(\Omega)$, the closure in $H^m(\Omega)$ of $C^{\infty}_c(\Omega)$. From the hypotheses on $H(\Omega)$ and the fact that $H^m(\Omega)\subset L^2(\Omega)$ is compact, it follows that problem \eqref{problemV} admits an increasing sequence of non-negative eigenvalues of finite multiplicity diverging to $+\infty$. 

We will denote by $\left\{\lambda_k^m\right\}_{k\in\mathbb N}$ the eigenvalues of \eqref{problemV} when $H(\Omega)=H^m_0(\Omega)$. Here we agree to repeat the eigenvalues according to their multiplicity. We will denote by $\left\{u_k^m\right\}_{k\in\mathbb N}$ the corresponding eigenfunctions, normalized by $\int_{\Omega}u_i^m u_j^mdx=\delta_{ij}$ for all $i,j\in\mathbb N$. It is easy to check that $\lambda_1^m>0$ for all $m\in\mathbb N$.

We will denote by $\left\{\mu_k^m\right\}_{k\in\mathbb N}$ the eigenvalues of \eqref{problemV} when $H(\Omega)=H^m(\Omega)$, where we agree to repeat the eigenvalues according to their multiplicity. We will denote by $\left\{v_k^m\right\}_{k\in\mathbb N}$ the corresponding eigenfunctions, normalized by $\int_{\Omega}v_i^m v_j^mdx=\delta_{ij}$ for all $i,j\in\mathbb N$. It is easy to check that $\mu_1^m=\cdots=\mu_{n(d,m)}^m=0$ and $\mu_{n(d,m)+1}^m>0$ for all $m\in\mathbb N$, where $n(d,m)$ denotes the dimension of the space of polynomials of degree at most $m-1$ in $\mathbb R^d$.

We will call $\left\{\lambda_k^m\right\}_{k\in\mathbb N}$ the {\it Dirichlet spectrum} of the polyharmonic operator $(-\Delta)^m$ and $\left\{\mu_k^m\right\}_{k\in\mathbb N}$ the {\it Neumann spectrum} of the polyharmonic operator $(-\Delta)^m$. In fact problem \eqref{problemV} is the weak formulation of
\begin{equation}
\begin{cases}
(-\Delta)^m u=\Lambda u, & {\rm in\ }\Omega,\\
u=\frac{\partial u}{\partial\nu}=\cdots=\frac{\partial^{m-1}u}{\partial\nu^{m-1}}=0, & {\rm on\ }\partial\Omega
\end{cases}
\end{equation}
when $H(\Omega)=H^m_0(\Omega)$, and of 
\begin{equation}
\begin{cases}
(-\Delta)^m u=\Lambda u, & {\rm in\ }\Omega,\\
\mathcal N_1 u=\cdots=\mathcal N_{m}u=0, & {\rm on\ }\partial\Omega
\end{cases}
\end{equation}
when $H(\Omega)=H^m(\Omega)$. Here $\mathcal N_i u$, $i=1,...,m$,  are uniquely defined complementing boundary operators of degree at most $2m-1$ which correspond to the classical Neumann boundary conditions for the polyharmonic operator. In particular, if $m=1$ then $\mathcal N_1u=\frac{\partial u}{\partial\nu}$; if $m=2$ then $\mathcal N_1 u=\frac{\partial^2 u}{\partial\nu^2}$ and $\mathcal N_2 u={\rm div}_{\partial\Omega}(D^2u\cdot\nu)_{\partial\Omega}+\frac{\partial\Delta u}{\partial\nu}$. It is in general a quite involved task to write the explicit form of the Neumann boundary conditions for $m\geq 3$.

The eigenvalues $\lambda_k^m$ admit the following variational characterization
\begin{equation}\label{minmaxD}
\lambda_k^m=\min_{\substack{W\subset H^m_0(\Omega)\\{\rm dim}(W)=k}}\max_{\substack{w\in W\\w\ne 0}}\frac{\int_{\Omega}|D^m w|^2dx}{\int_{\Omega}|w|^2dx},
\end{equation} 
and the eigenvalues $\mu_k^m$ admit the following variational characterization
\begin{equation}\label{minmaxN}
\mu_k^m=\min_{\substack{W\subset H^m(\Omega)\\{\rm dim}(W)=k}}\max_{\substack{w\in W\\w\ne 0}}\frac{\int_{\Omega}|D^m w|^2dx}{\int_{\Omega}|w|^2dx}.
\end{equation} 
From \eqref{minmaxD} and \eqref{minmaxN} it follows immediately that $\mu_k^m\leq\lambda_k^m$ for all $m,k\in\mathbb N$.

General inequalities between the eigenvalues of the Dirichlet and Neumann Laplacian on $\Omega\subset\mathbb R^d$, $d\geq 2$, have been widely studied in the past. The study of such inequalities was initiated by Payne \cite{payne_ineq} who proved that $\mu_{k+2}^1<\lambda_k^1$ for a bounded smooth convex domain in $\mathbb R^2$. He also conjectured in \cite{paynecomments} that in general $\mu_{k+1}^1\leq\lambda_k^1$ should hold (and even more, under additional assumptions like convexity). Generalization to higher dimension of Payne's result were achieved independently by Aviles \cite{aviles} and Levine-Weinberger \cite{levine_weinberger}. In particular, Aviles proved that $\mu_{k+1}^1<\lambda_k^1$ for smooth bounded domains in $\mathbb R^d$ with boundary of non-negative mean curvature. Levine and Weinberger established a family of inequalities of the form $\mu_{k+r}^1<\lambda_k^1$, $1\leq r\leq d$, under a series of assumptions on the sign of the principal curvatures of the boundary. In particular, if the domain is convex, $r=d$ (extending the result of Payne), while if the mean curvature is non-negative, $r=1$ (recovering Aviles' result). A definitive answer to Payne's conjecture was given by Friedlander \cite{friedlander_eigen} who proved $\mu_{k+1}^1<\lambda_k^1$ for a general smooth bounded domain in $\mathbb R^d$. An alternative proof (which does not require smoothness of the boundary) has been proposed in \cite{filonov}.

We will prove an analogue of Friedlander's result for the polyharmonic operator $(-\Delta)^m$ namely,

\begin{theorem}\label{main}
Let $\Omega$ be a bounded domain in $\mathbb R^d$, $d\geq 2$, such that the embedding $H^m(\Omega)\subset L^2(\Omega)$ is compact. Then
\begin{equation*}
\mu_{k+m}^m<\lambda_k^m
\end{equation*}
for all $m,k\in\mathbb N$. 
\end{theorem}

\begin{remark}
We remark that when $d=1$ it is easy to verify that $\mu^m_{k+m}=\lambda_k^m$ for all $m,k\in\mathbb N$. Let $\Omega=(0,1)$. The general solution of $(-1)^m u^{(2m)}(x)=\lambda u(x)$ with $\lambda\in\mathbb R$, $\lambda>0$, is given by $u(x)=\sum_{j=0}^{2m-1}\alpha_j \exp\left(e^{\frac{i\pi(2j+1)}{2m}}\lambda^{\frac{1}{2m}}x\right)$ for some $\alpha_0,...,\alpha_{2m-1}\in\mathbb C$. Dirichlet boundary conditions read $u(0)=u(1)=u'(0)=u'(1)=\cdots=u^{(m-1)}(0)=u^{(m-1)}(1)=0$,  while Neumann boundary conditions read $u^{(m)}(0)=u^{(m)}(1)=\cdots=u^{(2m-1)}(0)=u^{(2m-1)}(1)=0$. In both cases, boundary conditions produce a homogeneous system of $2m$ equations in $2m$ unknowns $\alpha_0,...,\alpha_{2m-1}$ which has a solution if and only if the determinant of the associated matrix is zero. Let us indicate by  ${\rm det}M_D(\lambda)$ ${\rm det}M_N(\lambda)$ the determinants of the systems produced by the Dirichlet and Neumann boundary conditions, respectively. All the positive Dirichlet (Neumann) eigenvalues are exactly the solutions of  ${\rm det}M_D(\lambda)=0$ (${\rm det}M_N(\lambda)=0$). It is straightforward to see that the two equations ${\rm det}M_D(\lambda)=0$ and ${\rm det}M_N(\lambda)=0$ have the same positive solutions. Moreover, by standard Sturm-Liouville theory, all the positive eigenvalues of both Dirichlet and Neumann problems have multiplicity one. The first Dirichlet eigenvalue $\lambda_1^m$ is positive. On the contrary, $0$ is an eigenvalue of the Neumann problem of multiplicity $m$, in fact the functions $1,x,...,x^{m-1}$ form a basis of the corresponding eigenspace. We conclude then that $\mu_{k+m}^m=\lambda_k^m$ for all $m,k\in\mathbb N$.
\end{remark}

\section{Proof of Theorem \ref{main}}

We will prove now Theorem \ref{main}. We prove first the following lemma.

\begin{lemma}\label{independence}
Let $\left\{\zeta_j\right\}_{j=1}^n\subset\mathbb C$ be a set of $n$ distinct complex numbers and let $\omega=(\omega_1,...,\omega_d)\in\mathbb R^d$, $\omega\ne 0$, be fixed. Then $\left\{e^{\zeta_j\omega\cdot x}\right\}_{j=1}^n\subset C^{\infty}(\mathbb R^d,\mathbb C)$ is a linearly independent set of functions.
\end{lemma}
\begin{proof}
Assume that there exist $\alpha_1,...,\alpha_n\in\mathbb C$ with $\alpha_j\ne 0$ for some $j$  such that $\sum_{j=1}^n\alpha_j e^{\zeta_j\omega\cdot x}=0$. % In particular $\partial^l_{x_{i_1}\cdots x_{i_l}}\left(\sum_{j=1}^n\alpha_j e^{\zeta_j\omega\cdot x}\right)=0$ for all $l\in\mathbb N$ and $(i_1,...,i_l)\in\left\{1,...,d\right\}^l$.
We can assume without loss of generality that $\omega_1\ne 0$. Let us differentiate $n-1$ times the equality $\sum_{j=1}^n\alpha_j e^{\zeta_j\omega\cdot x}=0$ with respect to $x_1$. We find that $\sum_{j=1}^n\alpha_j (\zeta_j w_1)^k e^{\zeta_j\omega\cdot x}=0$ for all $k=0,...,n-1$ and $x\in\mathbb R^d$. In particular, for $x=0$ we have $\sum_{j=1}^n\alpha_j (\zeta_j\omega_1)^k=0$ for all $k=0,...,n-1$. This is a system of $n$ equations in $n$ unknowns $\alpha_1,...,\alpha_n$. The associated matrix is a square Vandermonde matrix with non-zero determinant since all the $\zeta_j$ are distinct by hypothesis. Then necessarily $\alpha_1=\cdots=\alpha_n=0$, which contradicts the fact that $\alpha_j\ne 0$ for some $j$. This concludes the proof.
\end{proof}

We are now ready to prove Theorem \ref{main}.

\begin{proof}[Proof of Theorem \ref{main}]
Let $U$ denote the subspace of $H^m_0(\Omega)$ generated by the first $k$ Dirichlet eigenfunctions $u_1^m,...,u_k^m$. By hypothesis, ${\rm dim}(U)=k$.  Let us also denote by $\left\{\xi_j\right\}_{j=1}^m$  the set of the $m$-th roots of the unity, namely $\xi_j=e^{\frac{2j\pi i}{m}}$ for $j=0,...,m-1$. Let $\omega\in\mathbb R^d$ be a non-zero vector and let $V_{\omega}$ be the space generated by $\left\{e^{i\xi_j\omega\cdot x}{|_{\Omega}}\right\}_{j=1}^m$.\\
From Lemma \ref{independence} it follows that ${\rm dim}(V_{\omega})=m$ (we can assume without loss of generality that $0\in\Omega$). 

We observe now that $H^m_0(\Omega)\cap V_{\omega}=\left\{0\right\}$. In fact we can proceed as in the proof of Lemma \ref{independence} and assume that there exist $\alpha_1,...,\alpha_m$ with $\alpha_j\ne 0$ for some $j$ such that $v=\sum_{j=1}^m\alpha_j e^{i\xi_j\omega\cdot x}\in H^m_0(\Omega)$. Assume $\omega_1\ne 0$ and differentiate $m-1$ times the function $v$ with respect to $x_1$. We find that $\sum_{j=1}^m\alpha_j (i\xi_j w_1)^k e^{i\xi_j\omega\cdot x}=0$ for all $x\in\partial\Omega$ and $k=0,...,m-1$. Again, this is not possible unless $\alpha_1=\cdots=\alpha_m=0$. 

This proves that the sum $U+V_{\omega}$ is a direct sum for all choices of $\omega\ne 0$, hence ${\rm dim}(U+V_{\omega})=k+m$.  It is easy to verify that $V_{\omega}\cap V_z=\left\{0\right\}$ for all non-zero $\omega,z\in\mathbb R^d$, $\omega\ne z$ (see also Lemma \ref{independence}). Let $N\subset H^m(\Omega)$ be the space generated by all the Neumann eigenfunctions corresponding to the eigenvalue $\lambda_k^m$. If $\lambda_k^m$ is not a Neumann eigenvalue, we set $N=\left\{0\right\}$. Since all the eigenvalues have finite multiplicity, ${\rm dim}(N)<\infty$.

We note that $U\cap N=\left\{0\right\}$. This is trivial if $\lambda_k^m$ is not a Neumann eigenvalue since $N=\left\{0\right\}$. Otherwise, assume that there estists $u\ne 0\in U\cap N$. Hence
$$
\int_{\Omega}D^m u\cdot D^m \phi dx=\lambda_k^m\int_{\Omega}u\bar\phi dx\,,\ \ \ \forall\phi\in H^m(\Omega).
$$
Let us denote by $\tilde u$ the extension by $0$ of $u$ to $\mathbb R^d$ which belongs to $H^m(\mathbb R^d)$, since $u\in U\subset H^m_0(\Omega)$. Then
$$
\int_{\Omega}D^m \tilde u\cdot D^m \phi dx=\lambda_k^m\int_{\Omega}\tilde u\bar\phi dx\,,\ \ \ \forall\phi\in H^m(\mathbb R^d).
$$
This implies that $\tilde u$ is analytic in $\mathbb R^n$ and satisfies $(-\Delta)^m\tilde u=\lambda_k^m \tilde u$, and therefore $\tilde u=0$ in $\mathbb R^d$. In particular $u=0$, a contradiction.

We conclude then that there exists $\omega\in\mathbb R^d$ with $|\omega|^{2m}=\lambda_k^m$ such that $(U+V_{\omega})\cap N=\left\{0\right\}$. In fact, we have that $U\cap N=\left\{0\right\}$ , $V_{\omega}\cap V_z=\left\{0\right\}$ for all $z\ne\omega$ and $U\cap V_{\omega}=\left\{0\right\}$ for all $\omega$. Since we can choose any $\omega\in \mathbb R^d$ with $|\omega|^{2m}=\lambda_k^m$ and the space of such $\omega$ has infinite dimension (if $d\geq 2$), then we necessarily find a $\omega$ such that $(U+V_{\omega})\cap N=\left\{0\right\}$ (the space $N$ has finite dimension).

We set then $V:=V_{\omega}$. Now let $W:=U+V$. Clearly $W\subset H^m(\Omega)$ with ${\rm dim}(W)=k+m$. From \eqref{minmaxN} it follows immediately that
\begin{equation}\label{max}
\mu^m_{k+m}\leq\max_{\substack{w\in W\\w\ne 0}}\frac{\int_{\Omega}|D^m w|^2dx}{\int_{\Omega} |w|^2dx}.
\end{equation}
We will prove now that 
\begin{equation}
\int_{\Omega}|D^m w|^2dx\leq\lambda_k^m \int_{\Omega} |w|^2dx
\end{equation}
for all $w\in W$. Any $w\in W$ is of the form $w=u+v$ with $u\in U$ and $v\in V$. We have
\begin{multline}\label{chain0}
\int_{\Omega}|D^m w|^2dx= \int_{\Omega}\left(|D^m u|^2+|D^m v|^2 + 2{\rm Re} \left(D^m u\cdot D^m v\right)\right) dx\\
=\int_{\Omega}\left(|D^m u|^2+|D^m v|^2 + 2{\rm Re} \left(u (-\Delta)^m \bar v\right)\right) dx.
\end{multline}
The second equality is a consequence of the fact that $u\in H^m_0(\Omega)$ and of $m$ integrations by parts. From the definition of $U$ it follows that 
\begin{equation}\label{dir}
\int_{\Omega}|D^mu|^2dx\leq\lambda_k^m\int_{\Omega}|u|^2dx
\end{equation}
for all $u\in U$. Assume to know that 
\begin{equation}\label{quad}
|D^m v|^2=\lambda_k^m |v|^2
\end{equation}
 and 
\begin{equation}\label{poly}
(-\Delta)^mv=\lambda_k^mv
\end{equation}
for all $v\in V$. Then from \eqref{chain0}, \eqref{dir}, \eqref{quad} and \eqref{poly} it follows
\begin{multline}\label{chain1}
\int_{\Omega}|D^m w|^2dx\leq \lambda_k^m\int_{\Omega}\left(|u|^2+|v|^2 + 2{\rm Re} \left(u\bar v\right)\right) dx\\
=\lambda_k^m\int_{\Omega}|u+v|^2 dx=\lambda_k^m\int_{\Omega}|w|^2 dx,
\end{multline}
which implies $\mu_{k+m}^m\leq\lambda_{k}^m$ by \eqref{max}. Then, in order to prove $\mu_{k+m}^m\leq\lambda_{k}^m$ we have to prove \eqref{quad} and \eqref{poly}. Let $v\in V$, i.e., $v=\sum_{j=1}^m\alpha_j e^{i\xi_j\omega\cdot x}$ for some $\alpha_1,...,\alpha_m\in\mathbb C$. We have
\begin{equation}
-\Delta v=\sum_{j=1}^m\alpha_j \xi_j^2|\omega|^2 e^{i\xi_j\omega\cdot x},
\end{equation}
which immediately implies
\begin{equation}
(-\Delta)^m v=\sum_{j=1}^m\alpha_j \xi_j^{2m}|\omega|^{2m} e^{i\xi_j\omega\cdot x}=\lambda_k^m \sum_{j=1}^m\alpha_j e^{i\xi_j\omega\cdot x}=\lambda_k^m v,
\end{equation}
and hence the validity of \eqref{poly}. Moreover,
\begin{multline}
|D^m v|^2=\sum_{j_1,...,j_m=1}^d\left|\partial^m_{j_1\cdots j_m}v\right|^2=\sum_{j_1,...,j_m=1}^d\left|\sum_{j=1}^m\alpha_j i^m \xi_j^m\omega_{j_1}\cdots\omega_{j_m}e^{i\xi_j\omega\cdot x}\right|^2\\
=\sum_{j_1,...,j_m=1}^d\sum_{k,l=1}^m\alpha_k\overline{\alpha_l}\omega_{j_1}^2\cdots\omega_{j_m}^2 e^{i\xi_k\omega\cdot x}e^{-i\xi_l\omega\cdot x}\\
=\sum_{k,l=1}^m\alpha_k\overline{\alpha_l}e^{i\xi_k\omega\cdot x}e^{-i\xi_l\omega\cdot x}\sum_{j_1,...,j_m=1}^d\omega_{j_1}^2\cdots\omega_{j_m}^2\\
=|\omega|^{2m}\sum_{k,l=1}^m\alpha_k\overline{\alpha_l}e^{i\xi_k\omega\cdot x}e^{-i\xi_l\omega\cdot x}=\lambda_k^m\left|\sum_{j=1}^m\alpha_je^{i\xi_j\omega\cdot x}\right|^2=\lambda_k^m|v|^2,
\end{multline}
therefore \eqref{quad} holds.

We have then proved that $\mu_{k+m}^m\leq\lambda_k^m$ for all $m,k\in\mathbb N$. We observe that the inequality is actually strict. In fact, assume that $\mu_{k+m}^m=\lambda_k^m$ is a Neumann eigenvalue. Then equality holds in \eqref{max} and this implies that there exists a non-zero $u\in W\cap N$, but this is not possible since $W\cap N=\left\{0\right\}$ by construction. This concludes the proof.
\end{proof}

We conclude this note with a few remarks.

\begin{remark}
We remark that the proof of the strict inequality in Theorem \ref{main} does not work in dimension $d=1$ since we have only two choices of $\omega\in\mathbb R$ with $|\omega|^{2m}=\lambda_k^m$, and not necessarily $(U+V)\cap N=\left\{0\right\}$ (in dimension $d\geq 2$ we can choose any vector $\omega$ of the sphere of radius $(\lambda_k^m)^{\frac{1}{2m}}$ and hence we can always find a vector $\omega$ such that $(U+V)\cap N=\left\{0\right\}$). For example, let us consider the case of the Laplacian in one dimension. We can choose the space $V$ in the proof of Theorem \ref{main} either to be the space generated by $e^{i(\lambda_k^1)x}$ or the space generated by $e^{-i(\lambda_k^1)^{1/2}x}$. It is easy to see that there exists $v\in (U+V)\cap N$, $v\ne 0$. In fact, any element of $U+V$ is of the form $\sum_{j=1}^m\alpha_j\sin((\lambda_j^1)^{1/2}x)+\beta e^{i(\lambda_k^1)^{1/2}x}$ (we have chosen here the space $V$ generated by $e^{i(\lambda_k^1)^{1/2}x}$), for some $\alpha_1,...,\alpha_k,\beta\in\mathbb C$. Choose now $\alpha_j=0$ for $j=1,...,k-1$ and $\alpha_k=-i\beta$. Then $v=\cos((\lambda_k^1)^{1/2}x)$ which is clearly a non-zero function in $N$.
\end{remark}

\begin{remark}
We recall that the eigenvalue $0$ of the operator $(-\Delta)^m$ with Neumann boundary conditions has multiplicity $n(d,m)$, the dimension of the the space of polynomials in $\mathbb R^d$ of degree at most $m-1$. We have $n(d,m)>m$ for all $m\in\mathbb N$ and $d\geq 2$ (while $n(1,m)=m$ for all $m$). Hence it is natural to conjecture that actually $\mu^m_{n(d,m)+k}\leq\lambda_k^m$ for all $m,k\in\mathbb N$ (or even the strict inequality). If we want to repeat the proof of Theorem \ref{main} in order to obtain this last inequality, we should find for all $\lambda_k^m$ a $n(d,m)$-dimensional subspace $V$ of $H^m(\Omega)$ of functions such that, for all $v\in V$
\begin{enumerate}[i)]
\item $(-\Delta)^m v=\lambda_k^m v$ (in the $L^2(\Omega)$ sense);
\item $|D^mv|^2=\lambda_k^m|v|^2$ (in the $L^2(\Omega)$ sense);
\item $H^m_0(\Omega)\cap V=\left\{0\right\}$.
\end{enumerate}
If we want the strict inequality, we shall additionally require that $V+U$ is linearly independent with $N$ ($U,N$ are as in the proof of Theorem \ref{main}).

We note that such a space $V$ cannot be generated by $n(d,m)$ functions of the form $e^{\omega_j\cdot x}$ with $\omega_{j}\in\mathbb C^d$ for all $j=1,...,n(d,m)$, since in general we cannot find more that $m$ distinct $\omega_j\in\mathbb C^d$ such that $v=\sum_{j=1}^{n(d,m)}e^{\omega_j\cdot x}$ satisfy $i)$, $ii)$ and $iii)$ (one can easily verify such claim for $m=1$, $d=2$ and $n(2,2)=3$). We also note that actually less than $ii)$ is needed to conclude as in Theorem \ref{main}: it is sufficient that $\int_{\Omega}|D^mv|^2dx\leq\lambda_k^m\int_{\Omega}|v|^2dx$ for all $v\in V$, however this seems to be false in general if $V$ is generated by more than $m$ functions of the form $e^{\omega_j\cdot x}$.
\end{remark}

\begin{remark}
One may think to prove inequalities among Dirichlet and Neumann eigenvalues of the polyharmonic operators by proving inequalities among $\lambda_k^1$ and $\lambda_k^m$ and inequalities among $\mu_k^1$ and $\mu_k^m$ and combining them with Friedlander's result \cite{friedlander_eigen}.

Inequalities between Dirichlet eigenvalues of the Laplacian and of the polyharmonic operators can be easily obtained via min-max principle \eqref{minmaxD} and suitable integrations by parts. In particular, for any $u\in H^{m+1}_0(\Omega)$, $m\geq 1$ we have
\begin{equation}\label{m1}
\int_{\Omega}|D^mu|^2dx\leq\left(\int_{\Omega}|D^{m+1}u|^2dx\right)^{1/2}\left(\int_{\Omega}|D^{m-1}u|^2dx\right)^{1/2}.
\end{equation}
In order to prove formula \eqref{m1} we recall the following identity which holds for all $u\in H^m_0(\Omega)$ and which is obtained integrating by parts $\int_{\Omega}|D^mu|^2dx$ (see e.g., \cite{gazzola}):
$$
\int_{\Omega}|D^mu|^2dx=
\begin{cases}
\int_{\Omega}|\Delta^{\frac{m}{2}}u|^2dx\,, & {\rm if\ }m{\rm\ is\ even},\\
\int_{\Omega}|\nabla\Delta^{\frac{m-1}{2}}u|^2dx\,, & {\rm if\ }m{\rm\ is\ odd}.\\
\end{cases}
$$
Assume now that $u\in H^{m+1}_0(\Omega)$ (actually $u\in H^{m+1}(\Omega)\cap H^m_0(\Omega)$ is sufficient). Let $m$ be even. Then
\begin{multline*}
\int_{\Omega}|D^mu|^2dx=\int_{\Omega}|\Delta^{\frac{m}{2}}u|^2dx=-\int_{\Omega}\nabla\Delta^{\frac{m}{2}}u\cdot\nabla\Delta^{\frac{m-2}{2}}\bar u dx\\
\leq \left(\int_{\Omega}|\nabla\Delta^{\frac{m}{2}}u|^2dx\right)^{1/2}\left(\int_{\Omega}|\nabla\Delta^{\frac{m-2}{2}}u|^2dx\right)^{1/2}\\
=\left(\int_{\Omega}|D^{m+1}u|^2dx\right)^{1/2}\left(\int_{\Omega}|D^{m-1}u|^2dx\right)^{1/2},
\end{multline*}
where we have used  H\"older's inequality. This proves \eqref{m1} in the case that $m$ is even. The case $m$ odd is treated similarly. Inequality \eqref{m1} is then proved.

In particular, when $m=2$ we deduce from \eqref{m1} that
$$
\frac{\int_{\Omega}|\nabla u|^2dx}{\int_{\Omega}u^2dx}\leq\left(\frac{\int_{\Omega}|D^2u|^2dx}{\int_{\Omega}u^2dx}\right)^{1/2}
$$
By induction on $m$ it follows that for any $u\in H^{m+1}_0(\Omega)$,
$$
\left(\frac{\int_{\Omega}|D^m u|^2dx}{\int_{\Omega}u^2dx}\right)^{1/m}\leq\left(\frac{\int_{\Omega}|D^{m+1}u|^2dx}{\int_{\Omega}u^2dx}\right)^{1/m+1}.
$$
This implies, by \eqref{minmaxD}, that $\left(\lambda_k^m\right)^{1/m}$ is an increasing function of $m$ for all $k\in\mathbb N$, hence $(\lambda_k^1)^m\leq\lambda_k^m$ for all $m,k\in\mathbb N$ (the inequality is actually strict). Hence, from \cite{friedlander_eigen} it follows that $(\mu_{k+1}^1)^m<\lambda_k^m$ for all $k,m\in\mathbb N$. 

One wish to obtain now a reverse inequality between Neumann eigenvalues of the Laplacian and the polyharmonic operator in order to close the chain. However such inequalities are unavailable, and actually it is not clear whether a sort of monotonicity holds in the Neumann case. The inequality $\mu_k^m\leq (\mu_k^1)^m$  is trivially true only if $d=1$ and follows by testing eigenfunctions of the Neumann Laplacian into \eqref{minmaxN}. On the contrary, already for $m=2$ and $d\geq 2$ it is not at all understood if such an inequality always holds. 

We prove here, as an example, that the inequality $\mu_k^2\leq (\mu_k^1)^2$ holds for convex domains in $\mathbb R^d$, $d\geq 2$. In fact, for any $u\in H^2(\Omega)$ real-valued we have
\begin{equation}\label{part}
\int_{\Omega}|D^2u|^2dx=\int_{\Omega}(\Delta u)^2dx+\frac{1}{2}\int_{\partial\Omega}\partial_{\nu}(|\nabla u|^2)d\sigma-\int_{\partial\Omega}\Delta u\partial_{\nu}ud\sigma.
\end{equation}
Here $\nu$ denotes the outer unit normal to $\partial\Omega$ and $d\sigma$ the measure element of $\partial\Omega$. Equality \eqref{part} follows from the pointwise identity $|D^2u|^2=\frac{1}{2}\Delta(|\nabla u|^2)-\nabla\Delta u\cdot\nabla u$ which holds for smooth real-valued functions $u$. Now, we note that
\begin{multline}\label{part2}
\frac{1}{2}\partial_{\nu}(|\nabla u|^2)_{|_{\partial\Omega}}=\nabla (\partial_{\nu}u)\cdot\nabla u_{|_{\partial\Omega}}-D\nu\cdot\nabla u\cdot\nabla u\\
=\nabla_{\partial\Omega}(\partial_{\nu}u)\cdot\nabla u_{|_{\partial\Omega}}+\partial_{\nu}(\partial_{\nu}u)\nu\cdot\nabla u_{|_{\partial\Omega}}-II(\nabla_{\partial\Omega} u,\nabla_{\partial\Omega} u).
\end{multline}
Here $\nabla_{\partial\Omega}u_{|_{\partial\Omega}}$ denotes the tangential component of the gradient of $u$ on the boundary, and $II(\cdot,\cdot)$ the second fundamental form on $\partial\Omega$ (in fact $II=D\nu$).

Assume now that $u\in H^2(\Omega)$ is such that $\partial_{\nu}u=0$ on $\partial\Omega$ (in the sense of $L^2(\partial\Omega)$) and that $II\geq 0$ in the sense of quadratic forms. Then $\nabla u_{|_{\partial\Omega}}=\nabla_{\partial\Omega}u$ (the gradient of $u$ restricted on the boundary belongs to the tangent space to the boundary). This fact combined with \eqref{part} and \eqref{part2}  implies that for such $u$ and $\Omega$
$$
\int_{\Omega}|D^2u|^2dx\leq\int_{\Omega}(\Delta u)^2dx.
$$
Moreover, if $\Omega$ is a convex domain, then all eigenfunctions of the Neumann Laplacian belong to $H^2(\Omega)$ by standard elliptic regularity, and their normal derivatives vanish at the boundary (in $L^2(\partial\Omega)$). Hence, when $m=2$, taking into \eqref{minmaxN} as $k$-dimensional subspace of $H^2(\Omega)$ of test functions the space generated by the first $k$ eigenfunctions of the Neumann Laplacian, we immediately obtain $\mu_k^2\leq(\mu_k^1)^2$ for all $k\in\mathbb N$. 

However, also in the convex case, we cannot conclude more than $\mu_{k+1}^2<\lambda_k^2$, which is a weaker result with respect to $\mu_{k+2}^2<\lambda_k^2$ proved in Theorem \ref{main}.

\end{remark}

\section*{Acknowledgements}
The author is thankful to Dr. Davide Buoso for stimulating discussions - on inequalities between Dirichlet eigenvalues of polyharmonic operators - which were inspirational for the observations in the first part of Remark 2.4. The author would like to thank the anonymous referee for carefully reading the manuscript and for his/her useful comments and remarks.

\def\cprime{$'$} \def\cprime{$'$} \def\cprime{$'$} \def\cprime{$'$}
  \def\cprime{$'$}


\bibliography{bibliography}{}
\bibliographystyle{amsplain}

\end{document}